\documentclass{siamart171218}

\usepackage{amsfonts}
\usepackage{graphicx}
\usepackage{colortbl}

\newsiamremark{remark}{Remark}
\newsiamremark{hypothesis}{Hypothesis}
\crefname{hypothesis}{Hypothesis}{Hypotheses}
\newsiamthm{claim}{Claim}

\DeclareMathOperator{\nulls}{null}
\DeclareMathOperator{\range}{range}
\DeclareMathOperator{\rank}{rank}

\usepackage{array}
\newcolumntype{P}[1]{>{\centering\arraybackslash}p{#1}}

\headers{Preconditioning Systems with a Singular Leading Block}{Susanne Bradley}

\ifpdf
\hypersetup{
  pdftitle={Ideal Preconditioners for Saddle Point Systems with a Rank-Deficient Leading Block},
  pdfauthor={Susanne Bradley}
}
\fi

\title{Ideal Preconditioners for Saddle Point Systems with a Rank-Deficient Leading Block\thanks{\today}}
\author{Susanne Bradley\thanks{Department of Computer Science, University of British Columbia, Vancouver, BC, Canada. 
		(\email{smbrad@cs.ubc.ca}, \url{http://susannebradley.com}).}}
	
\usepackage{amsopn}

\begin{document}

\maketitle

\begin{abstract}
  We consider the iterative solution of symmetric saddle point systems with a rank-deficient leading block. We develop two preconditioners that, under certain assumptions on the rank structure of the system, yield a preconditioned matrix with a constant number of eigenvalues. We then derive some properties of the inverse of a particular class of saddle point system and exploit these to develop a third preconditioner, which remains ideal even when the earlier assumptions on rank structure are relaxed.
\end{abstract}



\section{Introduction}
Consider the saddle point linear system
\begin{equation}
\label{eqn:firstsp}
\underbrace{\begin{bmatrix}
	A & B^T \\
	B & 0
	\end{bmatrix}}_{=:\mathcal{K}}
\begin{bmatrix}
x \\
y
\end{bmatrix} =
\begin{bmatrix}
f \\
g
\end{bmatrix},
\end{equation}
where $A \in \mathbb{R}^{n \times n}$ is symmetric positive semidefinite, $B \in \mathbb{R}^{m \times n}$, and $m \le n$. We assume throughout this report that $B$ has full row rank (a necessary condition for $\mathcal{K}$ to be nonsingular).

When $A$ is positive definite, certain block Schur complement preconditioners yield preconditioned matrices with a constant number of eigenvalues \cite{murphy99}. Our goal in this report is to develop such ideal preconditioners when $A$ is singular. A common method for dealing with a singular $A$ is the \textit{augmented Lagrangian} approach (see, e.g., \cite{golub05}), which replaces $A$ with a positive definite augmented leading block of the form
$$
A + B^TW^{-1}B,
$$
where $W$ is a positive definite weight matrix. Our approach here is similar, but rather than augment with the entire matrix $B$, we will use only as many rows of $B$ as are necessary to alleviate the rank-deficiency of $A$.

We will see that we can sometimes develop preconditioners that are not ideal in general, but are ideal under certain assumptions on the rank structure of $\mathcal{K}$. Here we define two such assumptions that will be relevant in the upcoming discussion.

\begin{definition}[Maximal rank-deficiency]
	Let $\mathcal{K}$ be a nonsingular symmetric saddle point matrix with $A \in \mathbb{R}^{n \times n}$ and $B \in \mathbb{R}^{m \times n}$. We say that $\mathcal{K}$ has a \underline{maximally} \underline{rank-deficient leading block} if $\nulls(A) = m$.
\end{definition}

Because $\mathcal{K}$ is nonsingular if and only if $\ker(A) \cap \ker(B) = \emptyset$ \cite[Theorem 3.2]{benzi05}, having the nullity of $A$ be greater than $m$ would necessarily make $\mathcal{K}$ singular -- hence the term ``maximally rank-deficient.'' Estrin and Greif \cite{estrin15} develop preconditioners for matrices with this property.

\begin{definition}[Minimal independence]
	Let $\mathcal{K}$ be a nonsingular symmetric saddle point matrix whose leading block $A$ has nullity $m_2 = m - m_1$, for $m_1 \ge 0$. We say that $\mathcal{K}$ has \underline{minimally independent off-diagonal blocks} if $B$ contains $m_2$ row vectors outside the row space of $A$ and $m_1$ vectors within the row space. We write this as a $3 \times 3$ blocking of $\mathcal{K}$:
	\begin{equation}
	\label{eqn:minind_split}
	\mathcal{K} = \begin{bmatrix}
	A & B_1^T & B_2^T \\
	B_1 & 0 & 0 \\
	B_2 & 0 & 0
	\end{bmatrix},
	\end{equation}
	where the rows of $B_1 \in \mathbb{R}^{m_1 \times n}$ are in the row space of $A$ and those of $B_2 \in \mathbb{R}^{m_2 \times n}$ are outside the row space of $A$.
\end{definition}

Assuming that $B_1$ contains the rows that are linearly dependent on $A$ incurs no loss of generality because, for any $\mathcal{K}$ that satisfies the minimal independence condition, we can permute $B$ so that this happens. We call this ``minimal independence'' because if there were any fewer linearly independent rows in $B_2$, then $\mathcal{K}$ would be singular. Matrices with a maximally rank-deficient leading block are a special case, with $m_1$ (the number of rows in the dependent block $B_1$) being zero. Notice also that this definition guarantees that the submatrix
$$
\begin{bmatrix}
A & B_2^T \\ B_2 & 0
\end{bmatrix}
$$
is invertible.

In \cref{sec:prec_diag}, we present two preconditioners that are ideal in the cases of maximal rank-deficiency of $A$ and minimal independence of $B$. \Cref{sec:inv} explores some unique properties of matrices with minimally independent $B$, which we then exploit in \cref{sec:prec_gen} to develop a $3 \times 3$-block preconditioner that remains ideal even when the minimal independence assumption no longer holds. \Cref{tab:summary} summarizes the three preconditioners presented in this report.

\begin{table}[tbh!]
\centering
\begin{tabular}{ p{2.9cm} | P{2.7cm} | P{2.7cm} | P{2.7cm} |}
	& \multicolumn{3}{| c |}{{\cellcolor[gray]{0.9}}\textbf{Preconditioner}}\\
	& $\mathcal{P}_{2,D}$ & $\mathcal{P}_{3,D}$ & $\mathcal{P}_{3,T}$ \\
	\multicolumn{4}{| l |}{{\cellcolor[gray]{0.9}}\textbf{Conditions under which $\mathcal{P}$ is ideal}} \\
	\multicolumn{1}{|c|}{Assumed null$(A)$:} & $m$ & less than $m$ & less than $m$ \\
	\hline
	\multicolumn{1}{|c|}{Assumptions on $B$:} & none & minimal independence & none \\
	\multicolumn{4}{| l |}{{\cellcolor[gray]{0.9}}\textbf{Properties of $\mathcal{P}$}} \\
	\multicolumn{1}{|c|}{$\mathcal{P}$ positive definite?} & yes & yes & no \\
	\hline
	\multicolumn{1}{|c|}{Need to split $B$?} & no & yes & yes \\
	\hline
	\multicolumn{1}{|c|}{Need $\ker(A)$ basis?} & no & no & yes (or $\ker(B_2)$) \\
	\hline
	\multicolumn{1}{|c|}{Described in:} & \Cref{sec:prec_maxrd} (eq. \cref{eqn:pd2}) & \Cref{sec:prec_minind} (eq. \cref{eqn:pd3}) & \Cref{sec:prec_gen} (eq. \cref{eqn:p3t}) \\
	\hline
\end{tabular}
\caption{Summary of preconditioners for nonsingular saddle point systems with a singular leading block, where $A \in \mathbb{R}^{n \times n}$, $B \in \mathbb{R}^{m \times n}$.} \label{tab:summary}
\end{table}

\section{Preconditioners for special cases}
\label{sec:prec_diag}

\subsection{A preconditioner for maximal rank-deficiency}
\label{sec:prec_maxrd}
We assume here that $\nulls(A) =m$ (the number of rows in $B$). This means we must augment the leading block with all the rows of $B$ in order to make it full rank. For a positive definite weight matrix $W_B \in \mathbb{R}^{m \times m}$, we define the augmented system $\mathcal{K}_2(W_B)$ by
\begin{equation}
\label{eqn:k2w}
\mathcal{K}_2(W_B) := \begin{bmatrix}
A + B^T W_B^{-1}B & B^T \\
B & 0
\end{bmatrix}.
\end{equation}

We recall the following result about the Schur complement of the leading block in $\mathcal{K}_2(W_B)$ (see \cite[Theorem 3.5]{estrin15}):
\begin{proposition}
	\label{thm:schur_w}
	Suppose $\nulls(A) = m$ and let $W_B \in \mathbb{R}^{m \times m}$ be an invertible matrix. Then
	$$
	B(A+B^T W_B^{-1}B)^{-1}B^T = W_B.
	$$
\end{proposition}

Motivated by this result, we precondition $\mathcal{K}$ by taking the block diagonal Schur complement preconditioner \cite[equation (2)]{murphy99} of $\mathcal{K}_2(W_B)$, given by
\begin{equation}
\label{eqn:pd2}
\mathcal{P}_{2,D} := \begin{bmatrix}
A + B^TW_B^{-1}B & 0 \\
0 & W_B
\end{bmatrix}.
\end{equation}
The 2 in the subscript refers to the block size of the matrix and $D$ refers to the fact that the preconditioner is block diagonal. This particular preconditioner is also described in \cite{greif06}.

\begin{theorem}
	Suppose $\nulls(A) = m$ and let $W_B \in \mathbb{R}^{m \times m}$ be a positive definite matrix. The preconditioned matrix $\mathcal{P}_{2,D}^{-1}\mathcal{K}$ has eigenvalues $\lambda = 1, -1$, with respective geometric multiplicities $n$ and  $m$.
\end{theorem}
\begin{proof}
	The generalized eigenvalue problem is
	\begin{subequations}
	\label{eqn:eig_pd2}
	\begin{alignat}{2}
	\label{eqn:eig_pd2_x}
	A&x + B^Ty &&= \lambda Ax + \lambda B^TW_B^{-1}Bx; \\
	\label{eqn:eig_pd2_y}
	B&x &&= \lambda W_B y.
	\end{alignat}
	\end{subequations}
	We immediately obtain $n$ eigenvectors by observing that \cref{eqn:eig_pd2} is satisfied for $\lambda = 1$ by vectors of the form
	$$
	\begin{bmatrix}
	x \\
	W_B^{-1}Bx
	\end{bmatrix}, \ \ \ x \in \mathbb{R}^{n}.
	$$
	Moreover, any vector of the form
	$$
	\begin{bmatrix}
	x \\
	-W_B^{-1}Bx
	\end{bmatrix}, \ \ \ x \in \ker(A),
	$$
	satisfies \cref{eqn:eig_pd2} for $\lambda = -1$. Because $\nulls(A) = m$, we have accounted for all $n+m$ eigenvectors of $\mathcal{P}_{2,D}^{-1}\mathcal{K}$.
\end{proof}

This preconditioner can be used when $A$ is not maximally rank-deficient, though it will no longer be ideal. We refer the reader to \cite{greif06} for analysis of this case.

\subsection{A preconditioner for minimal independence}
\label{sec:prec_minind}
We now assume that $B$ is minimally independent and that the $m$ rows of $B$ have been partitioned as in \cref{eqn:minind_split} --i.e., that with $\nulls(A) = m_2$,
$$
\mathcal{K} = \begin{bmatrix}
A & B_1^T & B_2^T \\
B_1 & 0 & 0 \\
B_2 & 0 & 0
\end{bmatrix},
$$
where the rows of $B_1 \in \mathbb{R}^{m_1 \times n}$ are linearly dependent on the rows of $A$ and those of $B_2 \in \mathbb{R}^{m_2 \times n}$ are linearly independent.

The matrix $A + B_2^TW^{-1}B_2$ will positive definite for any positive definite weight matrix $W \in \mathbb{R}^{m_2 \times m_2}$ (we omit the subscript to distinguish this from the $m \times m$ weight matrix $W_B$ in the previous subsection). We adapt the preconditioner $\mathcal{P}_{2,D}$ to this case by defining the $3 \times 3$ block diagonal preconditioner:
\begin{equation}
\label{eqn:pd3}
\mathcal{P}_{3,D} = \begin{bmatrix}
A + B_2^T W^{-1}B_2 & 0 & 0 \\
0 & \Big(B_1(A + B_2^T W^{-1}B_2)^{-1}B_1^T\Big)/2 & 0 \\
0 & 0 & W
\end{bmatrix}.
\end{equation}
When $m_1 = 0$ (i.e., $B$ contains no rows that depend on $A$), this preconditioner reduces to $\mathcal{P}_{2,D}$. We can view this preconditioner as applying $\mathcal{P}_{2,D}$ to the subsystem
$
\begin{bmatrix}
A & B_2^T \\
B_2 & 0
\end{bmatrix}
$
and then handling the additional block $B_1$ with a Schur complement of the augmented leading block. We will show that the scaling factor of $\frac{1}{2}$ for this Schur complement leads to fewer eigenvalues of the preconditioned operator.

\subsection{Spectral analysis for the minimally independent case}
\label{sec:prec_ind_eig}

Before presenting the spectral analysis of $\mathcal{P}_{3,D}^{-1}\mathcal{K}$, we establish some lemmas that will prove helpful. We begin by recalling the following result \cite[Corollary 2.1]{estrin16}:
\begin{lemma}
	\label{thm:eg16}
	For matrices $M, N \in \mathbb{R}^{n \times n}$, if $\rank(M) = r$, $\rank(N) = n-r$, and $\rank(M+N)=n$, then $(M+N)^{-1}M$ is a projector of rank $r$.
\end{lemma}

When $B$ is minimally independent, $\rank(A) = n-m_2$ and $\rank(B_2^TW^{-1}B_2) = m_2$. Therefore, we can use \cref{thm:eg16} to establish the following:
\begin{lemma}
	\label{thm:isNullOfA}
	Suppose that $\nulls(A) = m_2$, $B_2$ is a subset of $m_2$ rows of $B$ such that $\begin{bmatrix} A & B_2^T \\ B_2 & 0 \end{bmatrix}$ is invertible, and $W$ is positive definite. Then $(A+B_2^TW^{-1}B_2)^{-1}B_2^T$ is a null matrix of $A$.
\end{lemma}
\begin{proof}
	We can write
	\begin{align*}
	A(A+B_2^TW^{-1}B_2)^{-1}B_2^T &= (A+ B_2^TW^{-1}B_2 - B_2^TW^{-1}B_2)(A+B_2^TW^{-1}B_2)^{-1}B_2^T\\
	&= B_2^T - B_2^TW^{-1}B_2(A+B_2^TW^{-1}B_2)^{-1}B_2^T.
	\end{align*}
	By \cref{thm:eg16}, $(A+B_2^TW^{-1}B_2)^{-1}B_2^TW^{-1}B_2$ is a projector of rank $m_2$. Because
	$$
	B_2^TW^{-1}B_2(A+B_2^TW^{-1}B_2)^{-1} = \Big((A+B_2^TW^{-1}B_2)^{-1}B_2^TW^{-1}B_2\Big)^T,
	$$
	$B_2^TW^{-1}B_2(A+B_2^TW^{-1}B_2)^{-1}$ is also a projector of rank $m_2$. Moreover, it is clear that its range is a subset of $\range(B_2^T)$. Because both $B_2^T$ and $B_2^TW^{-1}B_2(A+B_2^TW^{-1}B_2)^{-1}$ have rank $m_2$, we conclude that the projector's range is in fact equal to $\range(B_2^T)$. Therefore,
	\begin{align*}
	A(A+B_2^TW^{-1}B_2)^{-1}B_2^T &= B_2^T - B_2^TW^{-1}B_2(A+B_2^TW^{-1}B_2)^{-1}B_2^T \\
	&=B_2^T - B_2^T \\
	&= 0.
	\end{align*} 
\end{proof}

This lets us make a more specific observation about these projectors.

\begin{lemma}
	\label{thm:pac_1}
	Suppose that $\nulls(A) = m_2$, $B_2$ is a subset of $m_2$ rows of $B$ such that $\begin{bmatrix} A & B_2^T \\ B_2 & 0 \end{bmatrix}$ is invertible, and $W$ is positive definite. Then 
	$$
	A(A+B_2^TW^{-1}B_2)^{-1} =: P_A
	$$
	is a projector onto $\range(A)$ along $\ker(B_2)$.
\end{lemma}
\begin{proof}
	Because \cref{thm:eg16} tells us that $P_A^T$ is a projector of rank $n-m_2$, we deduce that $P_A$ is as well. Notice that $P_A$ is equal to $A$ multiplied by a full-rank matrix, which means that  $\range(P_A) = \range(A)$.
	
	Because $\nulls(P_A) = m_2$, showing that the projection is along $\ker(B_2)$ is equivalent to showing that $P_AB_2^T = 0$; this follows from \cref{thm:isNullOfA}.
\end{proof}

With this last result established, we can state our main result for this subsection.

\begin{theorem}
	Suppose that $\nulls(A) = m_2$, with $B$ minimally independent and partitioned as in \cref{eqn:minind_split}, and let $W$ be positive definite. The preconditioned matrix $\mathcal{P}_{3,D}^{-1}\mathcal{K}$ has eigenvalues $\lambda = -1, 1, 2$,  with respective geometric multiplicities $m_1 + m_2$, $n-m_1$, and $m_1$.
\end{theorem}

\begin{proof}
	Define $\tilde{A}_W := A+B_2^TW^{-1}B_2$ and $S_B := B_1\tilde{A}_W^{-1}B_1^T$. The eigenvalues of $\mathcal{P}_{3,D}^{-1}\mathcal{K}$ are given by the generalized eigenvalue problem:
	\begin{subequations}
		\label{eqn:eig_p3d}
		\begin{alignat}{2}
		\label{eqn:eig_p3d_x}
		&Ax + B_1^Ty_1 + B_2^Ty_2 &&= \lambda \tilde{A}_W x; \\
		\label{eqn:eig_p3d_y1}
		&B_1x &&= \frac{\lambda}{2} S_{B} y_1; \\
		\label{eqn:eig_p3d_y2}
		&B_2x &&= \lambda W y_2.
		\end{alignat}
	\end{subequations}
	Obtaining $y_1 = \frac{2}{\lambda}S_{B}^{-1}B_1x$ from \cref{eqn:eig_p3d_y1}, $y_2 = \frac{1}{\lambda} W^{-1}B_2x$ from \cref{eqn:eig_p3d_y2}, substituting these values into \cref{eqn:eig_p3d_x}, and multiplying both sides from the left by $\tilde{A}_W^{-1}$ gives:
	\begin{equation*}
	\tilde{A}_W^{-1}Ax + \frac{2}{\lambda}\tilde{A}_W^{-1}B_1^T S_{B}^{-1}B_1x + \frac{1}{\lambda} \tilde{A}_W^{-1}B_2^T W ^{-1} B_2x = \lambda x.
	\end{equation*}
	Because $\tilde{A}_W^{-1}(A+B_2^TW^{-1}B_2) = I$, we replace $\tilde{A}_W^{-1}B_2^T W ^{-1} B_2$ by $I - \tilde{A}_W^{-1}A$. We then multiply all terms by $\lambda$ and rearrange to give:
	\begin{equation}
	\label{eqn:eig_proj}
	\lambda^2 x - \lambda \underbrace{\tilde{A}_W^{-1}A}_{=:\ P_1}x - \Big( \underbrace{(I- \tilde{A}_W^{-1}A)}_{=\ I-P_1} + 2\underbrace{\tilde{A}_W^{-1}B_1^T S_{B}^{-1} B_1}_{=:\ P_2}\Big)x = 0.
	\end{equation}
	\cref{thm:eg16} shows that $P_1$ is a projector. Consequently, $I-P_1$ is as well. Note also that $P_2=\tilde{A}_W^{-1}B_1^T S_B^{-1} B_1$ is a projector onto the range of $\tilde{A}_W^{-1}B_1^T$. We can now show the desired geometric multiplicities by considering $x$ in the ranges (or null spaces) of the projectors in \cref{eqn:eig_proj}.
	
	First, consider $x \in \text{range}(\tilde{A}_W^{-1}B_1^T)$. Clearly, $P_2 x = x$. And because all columns of $B_1^T$ are in $\range(A)$,
	\begin{equation}
	\label{eqn:rp2_rp1}
	\range(P_2) = \range(\tilde{A}_W^{-1}B_1^T) \subseteq \range(\tilde{A}_W^{-1}A) = \range(P_1).
	\end{equation} 
	Thus, $P_1x = x$ and $(I - P_1)x =0$, so \cref{eqn:eig_proj} becomes $\lambda^2 x - \lambda x -2x =0$. This yields the eigenvalues $-1$ and $2$, each with geometric multiplicity $m_1$.
	
	Next, consider $x \in \ker(A)$. Clearly, $P_1x = 0$ and $(I-P_1)x=x$. Notice that because the rows of $B_1$ are in the row space of $A$, we can write $B_1 = QA$ for $Q \in \mathbb{R}^{m_1 \times n}$. Thus,
	\begin{equation}
	\label{eqn:kerab1}
	\ker(A) \subseteq \ker(B_1),
	\end{equation}
	so $P_2x = 0$. \Cref{eqn:eig_proj} then becomes $\lambda^2 x - x =0$, which yields $\lambda = \pm 1$, each with geometric multiplicity $m_2$. 
	
	Lastly, let us consider the term $\Big((I-P_1) + P_2\Big)$. Notice that
	$$
	\range(I-P_1) = \ker(P_1) \subseteq \ker(P_2),
	$$
	where the second relation holds because of \cref{eqn:kerab1}. Thus, the range of $(I-P_1)$ does not overlap with that of $P_2$, so
	$$
	\rank\Big((I-P_1)+2P_2\Big) = \rank(I-P_1) + \rank(P_2) = m_2+m_1.
	$$
	Another consequence of the lack of overlap is that $\Big((I-P_1)+2P_2\Big)x = 0$ if and only if $(I-P_1)x = 0$ and $P_2x = 0$. Therefore, any $x \in \ker\Big((I-P_1)+2P_2\Big)$ also satisfies $P_1x = x$, which means that \cref{eqn:eig_proj} becomes $\lambda^2 x - \lambda x = 0$. Because $\mathcal{K}$ is nonsingular, $\lambda=0$ cannot be an eigenvalue; we therefore have $n-m_1-m_2$ additional eigenvectors corresponding to $\lambda = 1$.
	
	We have now accounted for all $n+m_1+m_2$ eigenvectors, which completes the proof.
\end{proof}

\begin{remark}
	We saw that, by scaling the $(2,2)$-block of $\mathcal{P}_{3,D}$ by $\frac{1}{2}$, we obtain a preconditioned operator with three eigenvalues. With minor modifications to the previous analysis, it is easy to show that other scaling factors yield four eigenvalues. In terms of the number of eigenvalues, there is no benefit in scaling the $(1,1)$- or $(3,3)$-blocks.
\end{remark}

\section{Block $3 \times 3$ inverse formulas for the saddle point matrix}
\label{sec:inv}
Estrin and Greif \cite{estrin15} showed that when $A$ is maximally rank-deficient, the inverse of $\mathcal{K}$ has some surprising properties. They then exploited these to develop a block triangular approximate inverse preconditioner. In this section we show that analogous properties hold when $B$ is minimally independent; indeed, recalling that maximal rank-deficiency of $A$ is a special case of minimal independence of $B$, many of the results in \cite{estrin15} can be re-derived as corollaries of the results presented here. Later (\cref{sec:prec_gen}), we derive a block triangular preconditioner based on these insights. While the resulting preconditioner is more expensive than the block diagonal preconditioners of \cref{sec:prec_diag}, it is more generalizable: it remains ideal for systems in which $A$ is not maximally rank-deficient and $B$ is not minimally independent.

We begin from two known expressions for the inverse of the standard $2 \times 2$ saddle point system,
$$
K_2 = \begin{bmatrix}
\mathcal{A} & \mathcal{B}^T \\
\mathcal{B} & 0
\end{bmatrix}.
$$
When $\mathcal{A}$ is invertible, we can write (see \cite[equation (3.4)]{benzi05})
\begin{equation}
\label{eqn:posdefA_inv}
K_2^{-1} = \begin{bmatrix} \mathcal{A}^{-1} - \mathcal{A}^{-1}\mathcal{B}^T S_{\mathcal{B}}^{-1} \mathcal{B} \mathcal{A}^{-1} & \mathcal{A}^{-1}\mathcal{B}^T S_{\mathcal{B}}^{-1} \\
S_{\mathcal{B}}^{-1}\mathcal{B}\mathcal{A}^{-1} & -S_{\mathcal{B}}^{-1}
\end{bmatrix},
\end{equation}
where $S_\mathcal{B} = \mathcal{B}\mathcal{A}^{-1}\mathcal{B}^T$. A second formula \cite[equation (3.8)]{benzi05} does not require $\mathcal{A}$ to be invertible, and instead makes use of a matrix $Z_{\mathcal{B}} \in \mathbb{R}^{n \times (n-m)}$ whose columns form a basis for the null space of $\mathcal{B}$:
\begin{equation}
\label{eqn:zaz_inv}
K_2^{-1} = \begin{bmatrix}
V_{\mathcal{B}} & (I-V_{\mathcal{B}}\mathcal{A}) \mathcal{B}^T (\mathcal{B}\mathcal{B}^T)^{-1} \\
(\mathcal{B}\mathcal{B}^T)^{-1}\mathcal{B}(I-\mathcal{A}V_{\mathcal{B}}) & -(\mathcal{B}\mathcal{B}^T)^{-1}\mathcal{B}(\mathcal{A}-\mathcal{A}V_{\mathcal{B}}\mathcal{A})\mathcal{B}^T(\mathcal{B}\mathcal{B}^T)^{-1}
\end{bmatrix},
\end{equation}
where $V_{\mathcal{B}} = Z_{\mathcal{B}}(Z_{\mathcal{B}}^T\mathcal{A}Z_{\mathcal{B}})^{-1}Z_{\mathcal{B}}^T$.

\subsection{Null-$B_2$ inverse formula}
\label{sec:inv_nb2}
We now return our attention to the system $\mathcal{K}$. Assume that $B$ is minimally independent and split as in \cref{eqn:minind_split}. We then partition $\mathcal{K}$ as
$$
\mathcal{K} := \left[\begin{array}{c c | c}
A & B_1^T & B_2^T \\
B_1 & 0 & 0 \\
\hline
B_2 & 0 & 0
\end{array}\right],
$$
and treat it as a $2 \times 2$ block system with a singular leading block. We apply formula \cref{eqn:zaz_inv} with
$$
\mathcal{A} = \begin{bmatrix}
A & B_1^T \\
B_1 & 0
\end{bmatrix}, \ \ \ \ \ \mathcal{B} = \begin{bmatrix}
B_2 & 0
\end{bmatrix}.
$$
For the null matrix $Z_{\mathcal{B}}$ of $\mathcal{B}$, we use
$$
Z_{\mathcal{B}} = \begin{bmatrix}
Z_{B2} & 0 \\
0 & I
\end{bmatrix},
$$
where $Z_{B2} \in \mathbb{R}^{n \times (n-m_2)}$ is a matrix whose columns form a basis for the null space of $B_2$ and $I$ denotes the identity matrix of size $m_1$. Computing $V_{\mathcal{B}}$ requires the inverse of
$$
Z_{\mathcal{B}}^T\mathcal{A}Z_{\mathcal{B}} = \begin{bmatrix}
Z_{B2}^T AZ_{B2} & Z_{B2}^T B_1^T \\
B_1Z_{B2} & 0
\end{bmatrix}.
$$
Because the rows of $B_2$ are outside the row space of $A$, the matrix $\begin{bmatrix}
A & B_2^T \\
B_2 & 0
\end{bmatrix}$ is nonsingular. This means that $Z_{B2}^T AZ_{B2}$ is positive definite \cite[p. 20]{benzi05}, and that $Z_{\mathcal{B}}^T\mathcal{A}Z_{\mathcal{B}}$ is a saddle point system with a positive definite leading block. Thus, we can compute its inverse using \cref{eqn:posdefA_inv} and substitute this back into formula \cref{eqn:zaz_inv} to obtain
\begin{subequations}
	\label{eqn:Kinv}
	\begin{align}
	\mathcal{K}^{-1} = \left[\begin{array}{c c c}
	X_1 & X_2^T & X_3^T \\
	X_2 & X_4 & X_5^T \\
	X_3 & X_5 & X_6
	\end{array}
	\right],
	\end{align}
	\begin{align}
	\nonumber
	\text{with}&\\
	X_1 &:= V - V B_1^T S_{V}^{-1}B_1 V \\
	X_2 &:= S_{V}^{-1}B_1 V\\
	\label{eqn:x3}
	X_3 &:= (B_2B_2^T)^{-1}B_2\bigg( I - AV + AV B_1^TS_V^{-1}B_1 V - B_1^TS_V^{-1}B_1 V \bigg)\\
	X_4 &:= -S_{V}^{-1}\\
	\label{eqn:x5}
	X_5 &:= (B_2 B_2^T)^{-1}B_2(I-AV)B_1^TS_{V}^{-1} \\
	\label{eqn:x6}
	X_6 &:= -(B_2 B_2^T)^{-1}B_2\bigg( A + B_1^TS_{V}^{-1}B_1 -AV A - A V B_1^TS_{V}^{-1}B_1 - \\
	\nonumber
	 &\ \ \ \ \ \ B_1^TS_{V}^{-1}B_1 V A +A V B_1^TS_{V}^{-1}B_1 V A \bigg) B_2^T (B_2B_2^T)^{-1},
	\end{align}
\end{subequations}
where $V := Z_{B2}(Z_{B2}^T A Z_{B2})^{-1}Z_{B2}^T$ and $S_{V} := B_1V B_1^T$.

\Cref{eqn:Kinv} holds for any symmetric nonsingular $\mathcal{K}$ such that $\begin{bmatrix}
A & B_2^T \\ B_2 & 0
\end{bmatrix}$ is invertible. But fortunately for us, the assumption that $B$ is minimally independent lets us simplify some of the uglier terms. We begin with another lemma featuring the projection matrix $P_A$ of \cref{thm:pac_1}.

\begin{lemma}
	\label{thm:pac_2}
	If $\nulls(A) = m_2$ and $B_2$ is a subset of $m_2$ rows of $B$ such that $\begin{bmatrix}
	A & B_2^T \\ B_2 & 0
	\end{bmatrix}$ is invertible, then $AV = P_A$, where $P_A$ is a projector onto $\range(A)$ along $\ker(B_2)$.
\end{lemma}
\begin{proof}
	By writing 
	$$
	AV = AZ_{B2}(Z_{B2}^T A Z_{B2})^{-1}Z_{B2}^T,
	$$
	we see that $AV$ is a projector onto $\range(AZ_{B2})$ along $\ker(B_2)$. It is clear that $\range(AZ_{B2}) \subseteq \range(A)$; moreover, $A$ has rank $n-m_2$ and $AZ_{B2}$ has rank $n-m_2$. Because the range of $AZ_{B2}$ is a subset of the range of $A$, and both ranges have the same dimension, we conclude that $\range(AZ_{B2}) = \range(A)$.
\end{proof}
\begin{theorem}
	\label{thm:zeroblks}
	Let $\mathcal{K}$ have minimally independent off-diagonal blocks, with $B$ partitioned as in \cref{eqn:minind_split}. Then the (2,3)-, (3,3)- and (3,2)-blocks of $\mathcal{K}^{-1}$ are equal to zero, i.e.:
	$$
	\mathcal{K}^{-1} = \begin{bmatrix}
	\times & \times & \times \\
	\times & \times & 0 \\
	\times & 0 & 0
	\end{bmatrix}.
	$$
\end{theorem}
\begin{proof}
	We begin with the (2,3)-block of $\mathcal{K}^{-1}$, denoted by $X_5$ (eq. \cref{eqn:x5}).
	Because the columns of $B_1^T$ are in the range of $A$, \cref{thm:pac_2} implies that $AV B_1^T = B_1^T$, which gives $(I-AV)B_1^T=0$. Hence, $X_5 =0$.
	
	Now consider the (3,3)-block, $X_6$ (eq. \cref{eqn:x6}). Using $AVA = A$ and $AVB_1^T = B_1^T$, we find that the middle multiplicative term of $X_6$ (between the large parentheses) simplifies to zero, giving $X_6 = 0$.
\end{proof}

We can also simplify the $(3,1)$-block.
\begin{proposition}
	\label{thm:x3}
	Let $\mathcal{K}$ be nonsingular with minimally independent off-diagonal blocks, with $B$ partitioned as in \cref{eqn:minind_split}. Then the (1,3)- and (3,1)-blocks of $\mathcal{K}^{-1}$ can be simplified by writing
	$$
	X_3 = (B_2 B_2^T)^{-1}B_2(I-AV).
	$$
\end{proposition}
\begin{proof}
The result follows from the expression for $X_3$ in \cref{eqn:x3} and the observation in the proof of \cref{thm:zeroblks} that $AV B_1^T = B_1^T$.
\end{proof}

Incorporating the results of \cref{thm:zeroblks} and \cref{thm:x3} into the original inverse formula \cref{eqn:Kinv}, we obtain the following simplified formula when $\mathcal{K}$ has minimally independent off-diagonal blocks:
\begin{equation}
\label{eqn:Kinv_nullB2}
\mathcal{K}^{-1} = \begin{bmatrix}
V - V B_1^T S_{V}^{-1}B_1 V & V B_1^TS_{V}^{-1} & (I-V A)B_2^T (B_2B_2^T)^{-1} \\
S_{V}^{-1}B_1V & -S_{V}^{-1} & 0 \\
(B_2B_2^T)^{-1}B_2(I-AV) & 0 & 0
\end{bmatrix}.
\end{equation}
Because this expression uses a null space of $B_2$ in computing $V$, we refer to it as the \textit{null-$B_2$ inverse formula}.

\subsection{Null-$A$ inverse formula}
\label{sec:inv_na}
The zero blocks in the inverse of $\mathcal{K}$ let us derive a simpler expression for $\mathcal{K}^{-1}$. We recall the following result \cite[Proposition 2.1]{golub03}:

\begin{proposition}
	\label{thm:aug}
	Define
	$$
	K_2 = \begin{bmatrix}
	\mathcal{A} & \mathcal{B}^T \\
	\mathcal{B} & 0
	\end{bmatrix} \ \ \ \textrm{and} \ \ \ 
	K_2(W) = \begin{bmatrix}
	\mathcal{A} + \mathcal{B}^TW^{-1}\mathcal{B} & \mathcal{B}^T \\
	\mathcal{B} & 0
	\end{bmatrix},
	$$
	where $\mathcal{A} \in \mathbb{R}^{n \times n}$, $\mathcal{B} \in \mathbb{R}^{m \times n}$, and $W \in \mathbb{R}^{m \times m}$. If $K_2$ and $K_2(W)$ are both invertible, then
	$$
	(K_2(W))^{-1} = K_2^{-1} - \begin{bmatrix}
	0 & 0 \\
	0 & W^{-1}
	\end{bmatrix}.
	$$
\end{proposition}
Therefore, when we augment the leading block of $\mathcal{K}$ with $B_2$, we obtain
\begin{equation}
\label{eqn:k_aug_inv}
\underbrace{\begin{bmatrix}
A+B_2^TW^{-1}B_2 & B_1^T & B_2^T \\
B_1 & 0 & 0 \\
B_2 & 0 & 0
\end{bmatrix}}_{=:\ \mathcal{K}(W)}{}^{{}^{{}^{{}^{{}^{{}^{{}^{\mbox{\scriptsize $-1$}}}}}}}} =
\mathcal{K}^{-1} - \begin{bmatrix}
0 & 0 & 0 \\
0 & 0 & 0 \\
0 & 0 & W^{-1}
\end{bmatrix}.
\end{equation}
If $W$ is positive definite, then $A + B_2^TW^{-1}B_2$ is positive definite. We can obtain the inverse of $\mathcal{K}(W)$ by partitioning the matrix into a $2 \times 2$ system and applying formula \cref{eqn:posdefA_inv}, or we can compute it directly by Gaussian elimination. The result is
\setlength{\belowdisplayskip}{0pt} \setlength{\belowdisplayshortskip}{0pt}
$$
(\mathcal{K}(W))^{-1} = \ \ \ \ \ \ \ \ \ \ \ \ \ \ \ \ \ \ \ \ \ \ \ \ \ \ \ \ \ \ \ \ \ \ \ \ \ \ \ \ \ \ \ \ \ \ \ \ \ \ \ \ \ \ \ \ \ \ \ \ \ \ \ \ \ \ \ \ \ \ \ \ \ \ \ \ \ \ \ \ \ \ \ \ \ \ \ \ \ \ \ \ \
$$
\setlength{\abovedisplayskip}{0pt} \setlength{\abovedisplayshortskip}{0pt}
\setlength{\belowdisplayskip}{2pt} \setlength{\belowdisplayshortskip}{2pt}
\begin{footnotesize}
$$
\begingroup
\setlength\arraycolsep{3pt}\begin{bmatrix}
\hat{A} - \hat{A}B_2^T \bar{S}^{-1}B_2 \hat{A} & (I-\hat{A}B_2^T \bar{S}^{-1}B_2)\tilde{A}_W^{-1}B_1^T S_B^{-1} & \hat{A}B_2^T \bar{S}^{-1} \\
S_B^{-1}B_1\tilde{A}_W^{-1}(I - B_2^T \bar{S}^{-1}B_2 \hat{A}) & -S_B^{-1} - S_B^{-1}B_1 \tilde{A}_W^{-1}B_2^T \bar{S}^{-1} B_2 \tilde{A}_W^{-1}B_1^T S_B^{-1} & S_B^{-1}B_1 \tilde{A}_W^{-1}B_2^T \bar{S}^{-1} \\
\bar{S}^{-1}B_2\hat{A} & \bar{S}^{-1}B_2 \tilde{A}_W^{-1} B_1^T S_B^{-1} & -\bar{S}^{-1}
\end{bmatrix},\endgroup
$$
\end{footnotesize}
\normalsize
where $\tilde{A}_W = A+B_2^TW^{-1}B_2$, $S_B := B_1 \tilde{A}_W^{-1}B_1^T$, $\hat{A} : = \tilde{A}_W^{-1} - \tilde{A}_W^{-1}B_1^T S_B^{-1} B_1 \tilde{A}_W^{-1}$, and $\bar{S} := B_2\hat{A}B_2^T$. This formula is valid for any $3 \times 3$ saddle point system with an invertible leading block $\tilde{A}_W$.

When $B$ is minimally independent, we can combine \cref{eqn:k_aug_inv} with the result of \cref{thm:zeroblks} that the (3,3)-block of $\mathcal{K}^{-1}$ is zero to conclude that
\begin{equation}
\label{eqn:sbarw}
\bar{S} = W.
\end{equation}
It turns out that some of the terms of $(\mathcal{K}(W))^{-1}$ are related to the null space of $A$ and that we can further simplify our inverse formula using these connections. First, recall from \cref{eqn:kerab1} that $\ker(A) \subseteq \ker(B_1)$. This lets us establish the following simplifying result:
\begin{proposition}
	\label{thm:za_conv}
	Let $\mathcal{K}$ be nonsingular with minimally independent off-diagonal blocks, with $B$ partitioned as in \cref{eqn:minind_split}. Then $\tilde{A}_W^{-1}B_2^T = \hat{A}B_2^T$ for any positive definite $W$.
\end{proposition}
\begin{proof}
	We can write
	\begin{align*}
	\hat{A}B_2^T &= \tilde{A}_W^{-1}B_2^T - \tilde{A}_W^{-1}B_1^T S_B^{-1}B_1 \tilde{A}_W^{-1}B_2^T \\
	&= (I - \tilde{A}_W^{-1}B_1^T S_B^{-1}B_1)\tilde{A}_W^{-1}B_2^T.
	\end{align*}
	Notice that $\tilde{A}_W^{-1}B_1^T S_B^{-1}B_1$ is a projector onto range$(\tilde{A}_W^{-1}B_1^T)$ along range$(B_1^T)$. Therefore, $(I - \tilde{A}_W^{-1}B_1^T S_B^{-1}B_1)$ is a projector onto $\ker(B_1)$. By \cref{thm:isNullOfA}, $\tilde{A}_W^{-1}B_2^T$ is a null matrix of $A$. Thus, it is also a null matrix of $B_1$, and is therefore unchanged by the projection.
\end{proof}

So far, all results presented hold for any positive definite $W$. The following proposition -- which is essentially the converse of \cref{thm:isNullOfA} -- shows how one specific choice of $W$ can lead to a simpler expression for $\mathcal{K}^{-1}$.

\begin{proposition}
	\label{thm:act_null}
	Let $Z_A \in \mathbb{R}^{n \times m_2}$ be a matrix whose columns form a basis of $\ker(A)$. This matrix can be written in the form
	$$
	Z_A =\tilde{A}_W^{-1}B_2^T
	$$ 
	with $W = B_2Z_A$.
\end{proposition}
\begin{proof}
	 We first show that $B_2Z_A$ is nonsingular. Assume that 
	 $$
	 (B_2Z_A)x = B_2(Z_Ax) = 0
	 $$
	 for some $x$. Clearly, $Z_Ax \in \ker(A)$. Moreover, $B_2Z_Ax = 0$ implies that $Z_Ax \in \ker(B_2)$. But, we assumed that $\begin{bmatrix} A & B_2^T \\ B_2 & 0 \end{bmatrix}$ is nonsingular, and the existence of a nonzero vector $Z_A x$ belonging to both null spaces would contradict this assumption. Therefore, $x$ must be zero, which means that $B_2Z_A$ is nonsingular.
	
	Next, the matrix $Z_A$ is equal to $\tilde{A}_W^{-1}B_2^T$ if and only if $\tilde{A}_WZ_A = B_2^T$. We now confirm that when $W = B_2Z_A$,
	\begin{equation*}
	\tilde{A}_WZ_A = (A + B_2^T(B_2Z_A)^{-1}B_2)Z_A = B_2^T
	\end{equation*}
	for any $Z_A$ such that $AZ_A=0$, as required.
\end{proof}

We can now use the previous results to simplify the expression for $\mathcal{K}^{-1}$. Let us denote the weight matrix by
\begin{equation}
\label{eqn:def_L}
L := B_2Z_A
\end{equation}
to distinguish it from the arbitrary weight matrix $W$. We assume without loss of generality that $L$ is symmetric positive definite. (We say ``without loss of generality'' because we can always make $L$ positive definite by taking $Z_A(B_2Z_A)^T$ as our null matrix of $A$, and defining $L = B_2Z_A(B_2Z_A)^T$.) We drop the subscript from the corresponding augmented matrix $\tilde{A}_L$ and define 
$$
\tilde{A} := \tilde{A}_L = A+B_2^T L^{-1}B_2.
$$
We can simplify the expression for $(\mathcal{K}(W))^{-1}$ as follows: replace all occurrences of $\bar{S}$ with $L$ (\Cref{eqn:sbarw}); $\tilde{A}^{-1}B_2^T$ or $\hat{A}B_2^T$ with $Z_A$ (\cref{thm:za_conv,thm:act_null}); and $B_1Z_A$ or $Z_A^TB_1^T$ with zero (\Cref{eqn:kerab1}). This gives
\begin{equation}
\label{eqn:Kinv_nullA}
\mathcal{K}^{-1} = \left[\begin{array}{c c c}
\hat{A} - Z_A L^{-1}Z_A^T & \tilde{A}^{-1}B_1^T S_B^{-1} & Z_A L^{-1} \\
S_B^{-1}B_1\tilde{A}^{-1} & -S_B^{-1} & 0 \\
L^{-1}Z_A^T & 0 & 0
\end{array}
\right].
\end{equation}
 We refer to equation \eqref{eqn:Kinv_nullA} as the \textit{null-A inverse formula}. Recalling that $Z_A = \tilde{A}^{-1}B_2^T$, we can alternatively write the leading block in multiplicative form:
\begin{equation}
\label{eqn:Kinv_nullA_mult}
\mathcal{K}^{-1} = \left[\begin{array}{c c c}
(I-\tilde{A}^{-1}B_1^TS_B^{-1}B_1)\tilde{A}^{-1}(I-B_2^TL^{-1}Z_A^T) & \tilde{A}^{-1}B_1^T S_B^{-1} & Z_A L^{-1} \\
S_B^{-1}B_1\tilde{A}^{-1} & -S_B^{-1} & 0 \\
L^{-1}Z_A^T & 0 & 0
\end{array}
\right].
\end{equation}

\section{A preconditioner for the general case}
\label{sec:prec_gen}
In this section, we propose a block triangular preconditioner based on our insights about $\mathcal{K}^{-1}$ from \cref{sec:inv}. Though these analyses assumed that $B$ was minimally independent, we will combine elements of the null-$B_2$ and null-$A$ formulas to obtain a preconditioner that remains ideal even when this assumption does not hold. While we continue to assume that $B$ is partitioned into $B_1 \in \mathbb{R}^{m_1 \times n}$ and $B_2 \in \mathbb{R}^{m_2 \times n}$, and that $\begin{bmatrix} A & B_2^T \\ B_2 & 0 \end{bmatrix}$ is nonsingular, we no longer assume that the rows of $B_1$ are in the row space of $A$.

The following result illustrates a connection between the null-$B_2$ and null-$A$ formulas.
\begin{proposition}
	\label{thm:v}
	Suppose that $\nulls(A) = m_2$ and $B_2$ is a subset of $m_2$ rows of $B$ such that $\begin{bmatrix}
	A & B_2^T \\ B_2 & 0
	\end{bmatrix}$ is invertible. Then
	$$
	\tilde{A}^{-1}(I-B_2^TL^{-1}Z_A^T) = V,
	$$
	where $V=Z_{B2}(Z_{B2}^TAZ_{B2})^{-1}Z_{B2}^T$.
\end{proposition}
\begin{proof}
	Notice that $(I-B_2^TL^{-1}Z_A^T)$ is yet another expression for $P_A$, the projector onto $\range(A)$ along $\ker(B_2)$. So by \cref{thm:pac_2}, we can write
	\begin{equation}
	\label{eqn:pf1}
	\tilde{A}^{-1}(I-B_2^TL^{-1}Z_A^T) = \tilde{A}^{-1}AV.
	\end{equation}
	Because $A\tilde{A}^{-1} = P_A$ (by \cref{thm:pac_1}), $\tilde{A}^{-1}A = (A\tilde{A}^{-1})^T$ is a projector onto $\ker(B_2)$ along $\range(A)$. This means that $\tilde{A}^{-1}AZ_{B2} = Z_{B2}$. Therefore,
	$$
	\tilde{A}^{-1}AV = V,
	$$
	which, together with \cref{eqn:pf1}, completes the proof.
\end{proof}
For the preconditioner, we set  the $(2,3)$-, $(3,2)$-, and $(3,3)$-blocks equal to zero, in common with both inverse formulas. For the $(1,3)$- and $(3,1)$-blocks, we use the exact values of the null-$A$ formula. The four blocks in the top left we replace by a block diagonal Schur-complement-like operator, but we use $V$ in place of $\tilde{A}^{-1}$, yielding the $3 \times 3$ block triangular preconditioner:
\begin{equation}
\label{eqn:p3t}
\mathcal{P}_{3,T}^{-1} = \begin{bmatrix}
V & 0 & Z_A L^{-1} \\
0 & (B_1VB_1^T)^{-1} & 0 \\
L^{-1}Z_A^T & 0 & 0
\end{bmatrix}.
\end{equation}

To show that this matrix is well-defined, we note that $B_1VB_1^T$ is invertible for any nonsingular $\mathcal{K}$. Writing
$$
B_1VB_1^T = (B_1Z_{B2})\Big(Z_{B2}^TAZ_{B2}\Big)^{-1}(B_1Z_{B2})^T
$$
and recalling that $Z_{B2}^TAZ_{B2}$ is positive definite, we see that $B_1VB_1^T$ is positive definite if and only if $B_1Z_{B2}$ has full row rank. We prove this by showing that $Z_{B2}^TB_1^T$ has full column rank. Assume that $Z_{B2}^TB_1^Tx = 0$ for $x \ne 0$. Because $B_1^T$ is full-rank, this must mean that $B_1^T x \in \range(B_2^T)$, which in turn implies that $\range(B_1^T) \cap \range(B_2^T) \ne \emptyset$. This is a contradiction, as it implies that that $\begin{bmatrix} B_1^T & B_2^T \end{bmatrix}$ does not have full column rank, which would mean that $\mathcal{K}$ is singular.

We can now state the eigenvalues of the preconditioned matrix.

\begin{theorem}
	\label{thm:p3t_eig}
	Suppose that $\nulls(A) = m_2$ and $B_2$ is a subset of $m_2$ rows of $B$ such that $\begin{bmatrix}A & B_2^T \\ B_2 & 0 \end{bmatrix}$ is invertible. Then $\mathcal{P}_{3,T}^{-1}\mathcal{K}$ has eigenvalues $\lambda = 1, \frac{1+\sqrt{5}}{2}, \frac{1-\sqrt{5}}{2}$, with respective geometric multiplicities $n-m_1+m_2$, $m_1$, and $m_1$.
\end{theorem}

\begin{proof}
	The preconditioned operator is
	$$
	\mathcal{P}_{3,T}^{-1}\mathcal{K} = \begin{bmatrix}
	VA + Z_AL^{-1}B_2 & VB_1^T & 0 \\
	(B_1VB_1^T)^{-1}B_1 & 0 & 0 \\
	0 & L^{-1}Z_A^TB_1^T & I
	\end{bmatrix} = \begin{bmatrix}
	I & VB_1^T & 0 \\
	(B_1VB_1^T)^{-1}B_1 & 0 & 0 \\
	0 & L^{-1}Z_A^TB_1^T & I
	\end{bmatrix}.
	$$
	The second equality holds because $VA = (AV)^T = P_A^T$ (\cref{thm:pac_2}) and
	$$
	Z_AL^{-1}B_2 = I - (I-B_2^TL^{-1}Z_A^T)^T = I-P_A^T.
	$$
	Also note that we cannot simplify the $(3,2)$-block of $\mathcal{P}_{3,T}^{-1}\mathcal{K}$ to zero, as we are no longer assuming that $B_1$ is linearly dependent on $A$. 
	
	We write the eigenvalue equations as
	\begin{subequations}
	\label{eqn:eig_p3t}
	\begin{alignat}{3}
	\label{eqn:eig_p3t_x}
	x & \ + \ \  VB_1^Ty_1&  =&& \lambda x; \\
	\label{eqn:eig_p3t_y1}
	(B_1VB_1^T)^{-1}B_1x \ & \ & \ \ =&& \ \lambda y_1; \\
	\label{eqn:eig_p3t_y2}
	\ \ & L^{-1}Z_A^TB_1^Ty_1& \ + \ y_2 \ \ =&& \ \lambda y_2.
	\end{alignat}
	\end{subequations}	
	We immediately find that $\lambda = 1$ is an eigenvalue with multiplicity at least $n - m_1 +m_2$, with eigenvectors of the form
	$$
	\begin{bmatrix}
	x \\
	0 \\
	y_2
	\end{bmatrix} \  \textrm{for } x \in \ker(B_1), \ y_2 \in \mathbb{R}^{m_2}.
	$$
	For $\lambda \ne 1$, we can eliminate $y_2$ by setting
	$$
	y_2 = \frac{1}{\lambda-1} L^{-1}Z_A^TB_1^Ty_1,
	$$
	by \cref{eqn:eig_p3t_y2}. Next, we solve \cref{eqn:eig_p3t_y1} for $y_1$ and substitute into \cref{eqn:eig_p3t_x} to give
	$$
	x + \frac{1}{\lambda}VB_1^T(B_1 V B_1^T)^{-1}B_1x = \lambda x.
	$$
	Note that $VB_1^T(B_1 V B_1^T)^{-1}B_1$ is a rank-$m_1$ projector onto $\range(VB_1^T)$. Considering $x$ in this range, we obtain
	$$
	x + \frac{1}{\lambda}x = \lambda x,
	$$
	which gives $\lambda = \frac{1 \pm \sqrt{5}}{2}$, each with multiplicity $m_1$. Along with the $n-m_1+m_2$ eigenvalues equal to 1, we have accounted for all eigenvalues.
\end{proof}

\begin{remark}
	\label{thm:rmk1}
	There are two ways to compute $V$. In practice, $V$ needs to be approximated, and depending on the problem in question, one of these formulations may be easier to handle. If we have a way to approximately form and solve the augmented leading block, we could use
	$$
	\tilde{V} \approx (A+B_2^T L^{-1}B_2)^{-1}(I-B_2^T L^{-1}Z_A^T).
	$$
	Alternatively, if a null space of $B_2$ is available and we can approximately invert $Z_{B2}^TAZ_{B2}$, we could instead use
	$$
	\tilde{V} \approx Z_{B2}(Z_{B2}^TAZ_{B2})^{-1}Z_{B2}^T.
	$$
\end{remark}
\begin{remark}
	This preconditioner can also be used if a null space of $A$ is not readily available but a null space of $B_2$ is. We can use the second strategy of \cref{thm:rmk1} to approximate $V$ and replace the $(3,1)$- and $(1,3)$-blocks of $\mathcal{P}_{3,T}^{-1}$ by their null-$B_2$ values, giving the mathematically equivalent but more expensive formulation:
	$$
	\mathcal{P}_{3,T}^{-1} = \begin{bmatrix}
	V & 0 & (I-VA)B_2^T(B_2B_2^T)^{-1} \\
	0 & (B_1VB_1^T)^{-1} & 0 \\
	(B_2B_2^T)^{-1}B_2(I-AV) & 0 & 0
	\end{bmatrix}.
	$$
\end{remark}
\begin{remark}
	If $m_1=0$ and $m_2 = m$ (i.e., $A$ is maximally rank-deficient), then this preconditioner is actually the exact inverse of $\mathcal{K}$. See \cite{estrin15} for proof, and for preconditioners derived by approximating the leading block.
\end{remark}

\section{Discussion and conclusions}
\label{sec:conclusions}
We have presented three ideal preconditioners for saddle point systems with a singular leading block. The two block diagonal preconditioners rely on assumptions about the rank structure of the matrix (maximal rank-deficiency or minimal independence), while the block triangular preconditioner does not. While the block diagonal preconditioners may be used when the rank structure assumptions do not hold, we will no longer have a constant number of eigenvalues in the preconditioned operator. We refer to \cite{greif06} for analysis of $\mathcal{P}_{2,D}$ in the general case. A corresponding investigation for $\mathcal{P}_{3,D}$ is a subject for future work.

In practice, all three precondtioners will be too expensive to apply exactly. All require inverting an augmented leading block and computing a Schur complement or Schur-complement-like expression. Additionally, the $3 \times 3$ block preconditioners require identifying a subset of rows of the off-diagonal block that are linearly independent of the leading block, which may be difficult in some cases, and the block triangular preconditioner requires a null space of $A$ or $B_2$. Making these preconditioners practical will require effective approximations to the more expensive terms. These approximation strategies will need to be informed by the problem to be solved.

\bibliographystyle{siamplain}
\bibliography{references}
\end{document}